\documentclass[10pt]{amsart}

\usepackage[utf8]{inputenc}
\usepackage[T1]{fontenc}

\usepackage{amsmath}
\usepackage{float}
\usepackage[all]{xy}
\usepackage{fancyhdr}
\usepackage{graphicx}

\usepackage{mathtools}
\usepackage{adjustbox}
\usepackage{enumitem}
\usepackage{stackrel}
\usepackage{scalerel}
\usepackage{euscript}
\usepackage{amsmath}
\usepackage{amsthm}
\usepackage{indentfirst}

\usepackage{mathpazo}
\usepackage{mathrsfs}
\usepackage{amsmath}
\usepackage{amssymb}
\usepackage{hyperref}
\usepackage{cleveref}
\usepackage{comment}
\usepackage{color}
\usepackage{tikz}
\usepackage{tikz-cd}
\usepackage{xparse}
\usetikzlibrary{matrix,arrows,decorations.pathmorphing}

\newtheorem{thm}{Theorem}[section]
\newtheorem{prop}[thm]{Proposition}
\newtheorem{cor}[thm]{Corollary}
\newtheorem{lem}[thm]{Lemma}

\newtheorem{remark}[thm]{Remark}

\crefname{lem}{Lemma}{lemma}
\crefname{remark}{Remark}{remark}
\crefname{cor}{Corollary}{corollary}
\crefname{thm}{Theorem}{theorem}
\crefname{prop}{Proposition}{proposition}
\crefname{example}{Example}{example}
\crefname{defn}{Definition}{definition}
\crefname{notation}{Notation}{notation}
\crefname{appendix}{Appendix}{appendix}
\crefname{section}{Section}{section}

\newcommand{\AAA}{\mathcal A}

\newcommand{\NNN}{\mathcal N}
\newcommand{\KKK}{\mathcal K}
\newcommand{\BBB}{\mathcal B}
\newcommand{\HHH}{\mathcal H}
\newcommand{\LLL}{\mathcal L}

\newcommand\Nb {\mathbb{N}}

\newcommand\Cb {\mathbb{C}}

\usepackage{amsfonts} 
\DeclareMathSymbol{\shortminus}{\mathbin}{AMSa}{"39}
\usepackage{mathpazo}

\DeclareMathOperator{\Alg}{Alg}

\DeclareDocumentCommand{\lin}{m O{\cdot} O{\cdot}}{{}_{{#1}{ }} \langle #2, #3 \rangle}
\DeclareDocumentCommand{\rin}{m O{\cdot} O{\cdot}}{\langle #2, #3 \rangle_{{#1}{}} }

\DeclarePairedDelimiterX\braket[2]{\langle}{\rangle}{#1 \delimsize\vert #2}

\newcommand{\titleinfo}{Local derivation on some class of subspace lattice algebras}
\newcommand{\titleinfoshort}{local derivation}
\thanks{This work was   supported  by the National Natural Science Foundation of China (No. 12471124).}
\newcommand{\authorinfo}{}

\begin{document}

\title{\LARGE\textbf{\titleinfo}}

\author[Chen]{Hongjie Chen}

\address{School of Mathematical Sciences, Qufu Normal University, Qufu, Shandong, 273165, China}

\email{chenhongjiehei@163.com}

\author[Wang]{  Liguang Wang}
\address{School of Mathematical Sciences, Qufu normal University, Qufu 273165, China.}
\email{wangliguang0510@163.com}

\author[Yang]{ Zhujun Yang}
\address{School of Mathematical Sciences, Qufu Normal University, Qufu, Shandong, 273165, China}
\email{zjyang2100@163.com}

\begin{abstract}
Let $\HHH$ be a separable Hilbert space and $\LLL_{0}\subset\BBB(\HHH)$ a complete reflexive lattice. Let   $\mathscr{K}$ be the direct sum of   $n_0$ copies of $\HHH$ ($n_{0}\in\Nb$ and $n_0\geq 2$) or the direct sum of countably infinite many copies of $\HHH$ respectively.  We construct two class of subspace lattices  $\LLL$ on $\mathscr{K}$.  Let $\Alg\LLL$ be the corresponding subspace lattice algebra.  We show that every local derivation from $\Alg\LLL $ into $\BBB(\mathscr{K})$ is a derivation. 
\end{abstract}

\date{}
\maketitle

\section{Introduction}

Let $\AAA$ be an algebra. A linear map $D: \AAA\to \AAA$ is said to be a  derivation  if
$$
D(ab) = D(a)b + aD(b)
$$ for all $a,b\in \AAA$. We call a  map
 $\delta: \AAA\to \AAA$   a  {local derivation} (see \cite{LoK}) if
for any $a$ in $\AAA$ there is a  derivation $D_a$ of $\AAA$, depending on $a$, such that
$
\delta(a) = D_a(a).
$
Every linear local derivation of a $C^\ast$-algebra is continuous (\cite{Johnson01}), and
indeed a derivation (\cite{shulman94}; see also \cite{bresar92}).
In particular, every  linear  local derivation $\delta$ on a von Neumann algebra  is  an inner derivation (\cite{Kad}). They are also many results on derivations and local derivations on 
non-self-adjoint operator algebras which is parallel to  
self-adjoint algebras.  We refer to \cite{LoHW,WuJ,WuJ2,ZhuJ}  
for more results on local derivations on non-self-adjoint operator algebras.

Now we recall some backgrounds of non-self-adjoint operator algebras. A subset $\LLL$ of subspaces of a Hilbert space $\HHH$ is called a subspace lattice if it contains $\{0\}$ and $\HHH$ and is complete in the sense that it is closed under the formation of arbitrary closed linear spans(denoted by $\vee$) and intersections(denoted by $\wedge$). Let $\LLL$ be a subspace lattice. The subspace lattice algebra $\Alg\LLL$ corresponds to $\LLL$ is defined by
$$
\Alg\LLL=\{A\in\BBB(\HHH):Ax\in E \ {\rm{for \ all \ }} x\in E \ {\rm{and \ }}E\in\LLL\}.
$$

For any $P\in\LLL$, the immediate predecessor and the immediate successor of $P$ are defined by
$$
P_{-}=\vee\{Q\in\LLL:Q\ngeqslant P\} \ {\rm{for}} \ P\neq 0, \ {\rm{and}} \ 0_{-}=0,
$$
and
$$
P_{+}=\wedge\{Q\in\LLL:Q\nleqslant P\} \ {\rm{for}} \ P\neq I, \ {\rm{and}} \ I_{+}=I
$$
respectively. If $P$ is also in other lattice $\LLL_{1}$, we will denote by $P_{-}^{\LLL_{1}}$ and $P_{+}^{\LLL_{1}}$ the "predecessor" and "successor" of $P$ in $\LLL_{1}$, respectively.

For a  subalgebra $\AAA$ of $\BBB(\HHH)$, ${\rm{Lat}}\AAA$ is  the lattice of subspaces of $\HHH$ that are left invariant by each operator in $\AAA$. An algebra $\AAA$ is reflexive if $\AAA=\Alg{\rm{Lat}}\AAA$, and a lattice $\LLL$ is reflexive if $\LLL={\rm{Lat}}\Alg\LLL$.

Let $\LLL$ be a subspace lattice. For convenience, we put
$$
{\rm{I_{d}}}(\Alg\LLL)={\rm{span}}\{T\in\Alg\LLL:T \ {\rm{is \ a \ idempotent \ operator}}\},
$$
and
$$
{\rm{R_1}}(\Alg\LLL)={\rm{span}}\{T\in\Alg\LLL:T \ {\rm{is \ of \ rank}} \ 1\}.
$$

Now we introduce some class of the reflexive subspace lattices that we will consider in this paper (see \cite{Wu}). The following notations will be fixed throughout this paper.

Let $\LLL_{0}\subset\BBB(\HHH)$ be a complete reflexive lattice. Let $n_{0}\in\Nb$ and $n_0\geq 2$. Let $\KKK_{n_{0}}=\HHH^{(n_{0})}$ be the direct sum of $n_0$ copies of $\HHH$. Suppose $$\zeta=(a_{1},\cdots,a_{n_{0}}) \in \Cb^{n_0}$$ is a unit vector  with $a_{i}\neq 0$ for any $1\leq i\leq n_{0}$. Let $Q_{\zeta}$ be the orthogonal projection of $\KKK_{n_{0}}$ onto the closed subspace spanned by  $\{(a_{1}x,\cdots,a_{n}x): x\in\HHH\}$. Let $\{E_{ij}\}$ be the standard matrix units in the matrix algebra $M_{n_0}(\Cb)$ and $F_{k}=\sum^{k}_{j=1}E_{jj}\otimes I$ for all $1\leq k\leq n_{0}$. Let $\LLL_{n_{0}}$ be the lattice generated by the projections $\{E_{11}\otimes P,Q_{\zeta},F_{k}:P\in\LLL_{0},2\leq k\leq n_{0}\}$. It is not difficult to see that
\begin{align}
\LLL_{n_{0}}=\{E_{11}\otimes P,Q_{\zeta},(E_{11}\otimes P)\vee Q_{\zeta},F_{k},F_{k}\vee Q_{\zeta}:P\in\LLL_{0},2\leq k\leq n_{0}\} \nonumber
\end{align}
and
\begin{align*}
\Alg\LLL_{n_{0}}&=\{\sum_{i=1,j\geq i}^{n_{0}}E_{ij}\otimes T_{ij}:T_{ij}\in\BBB(\HHH),T_{11}\in\Alg\LLL_{0},\\
&\sum_{j=1}^{n_{0}}\frac{a_{j}}{a_{1}}T_{1j}=\sum_{j=2}^{n_{0}}\frac{a_{j}}{a_{2}}T_{2j}=\cdots=T_{n_{0},n_{0}}\}.
\end{align*}
 
Now let $$\KKK_{\infty}=\oplus^{\infty}_{n=1}\HHH_{n}(=l^{2}(\Nb)\otimes\HHH)$$ with $\HHH_{n}=\HHH$ for any $n$. Suppose $$\eta=(b_{1},\cdots,b_{n},\cdots) \in l^{2}(\Nb)$$ is a unit vector  with $b_{n}\neq 0$ for any $n\geq 1$. Let $Q_{\eta}$ be the orthogonal projection of $\KKK_{\infty}$ onto the closed subspace $\{(b_{1}x,\cdots,b_{n}x,\cdots): x\in\HHH\}$. Let $\{E_{ij}\}$ be the standard matrix units of $\BBB(l^2(\Nb))$ and $E_{k}=\sum^{k}_{j=1}E_{jj}\otimes I$. Let $\LLL_{\infty}$ be the lattice generated by the projections $\{E_{11}\otimes P,Q_{\eta},F_{k}:P\in\LLL_{0},k\geq2\}$. Then we have
\begin{align}
\LLL_{\infty}=\{I,E_{11}\otimes P,Q_{\eta},(E_{11}\otimes P)\vee Q_{\eta},E_{k},E_{k}\vee Q_{\eta}:P\in\LLL_{0},k\geq 2\} \nonumber
\end{align}
and
\begin{align*}
\Alg\LLL_{\infty}&=\{\sum_{i=1,j\geq i}^{\infty}E_{ij}\otimes T_{ij}:T_{ij}\in\BBB(\HHH),T_{11}\in\Alg\LLL_{0},\\
&\sum_{j=1}^{\infty}\frac{a_{j}}{a_{1}}T_{1j}=\sum_{j=2}^{\infty}\frac{a_{j}}{a_{2}}T_{2j}=\cdots=\sum_{j=n}^{\infty}\frac{a_{j}}{a_{n}}T_{nj}=\cdots\}.
\end{align*}

It follows from \cite{Wu} that the lattices $\LLL_{n_{0}}$ and $ \LLL_{\infty}$ are reflexive subspace lattices when the lattice $\LLL_{0}$ is a reflexive subspace lattice.

For  simplicity, we will not distinguish an operator $T\in\BBB(\HHH)$ and its counterpart $E_{11}\otimes T\in\BBB(\KKK_{n_{0}})$ (respective, $\BBB(\KKK_{\infty})$) in the rest of this paper. We also identify  an orthogonal projection $P\in B(\KKK_{n_{0}})$ (respective, $\BBB(\KKK_{\infty})$) with its range $P(\KKK)$. We use $x\in P$ to indicate that a vector $x$ is in the range of a orthogonal projection $P$.

Our aim in this note is to show that local derivations on $\Alg\LLL_{n_0}$ and $\Alg\LLL_{\infty}$ are derivations.

\section{Main results}
We first describe the structure of the lattices given in section 1. 

\begin{prop} 
Let $\LLL_{n_{0}}$ be the  subspace lattice given in section 1. Then
\begin{itemize}
\item[(1)] $(E_{11}\otimes P)_{-}=(E_{11}\otimes P_{-}^{\LLL_{0}})\vee Q_{\zeta}$, $\forall P\in\LLL_{0}$;
\item[(2)] $(F_{k})_{-}=F_{k-1}\vee Q_{\zeta}$, $1<k<n_{0}$;
\item[(3)] $(Q_{\zeta})_{-}=F_{n_{0}-1}$;
\item[(4)] $I_{-}=I$, $0_{+}=0$.
\end{itemize}
\end{prop}

\begin{proof}
The conclusion follows easily from the definitions and construction of the lattice $\LLL_{n_0}$. We omit the details.  
\end{proof}

Similar  results also hold for  $\LLL_{\infty}$  and we omit its proof.

\begin{prop}
Let $\LLL_{\infty}$ be the  subspace lattice given in section 1. Then
\begin{itemize}
\item[(1)] $(E_{11}\otimes P)_{-}=(E_{11}\otimes P_{-}^{\LLL_{0}})\vee Q_{\eta}$, $\forall P\in\LLL_{0}$;
\item[(2)] $(F_{k})_{-}=F_{k-1}\vee Q_{\eta}$, $k\geq 2$;
\item[(3)] $(Q_{\eta})_{-}=I$, $I_{-}=I$, $0_{+}=0$.
\end{itemize}
\end{prop}

In the following we always assume that $\LLL\in \{\LLL_{n_0}, \LLL_{\infty}\}$. We also denote  $\mathscr{K}=\KKK_{n_{0}}$ when $\LLL=\LLL_{n_0}$ and $\mathscr{K}=\KKK_{\infty}$ when $\LLL=\LLL_{\infty}$.
 
\begin{prop}\label{rank 1}
${\rm{R_1}}(\Alg\LLL )\subset{\rm{I_d}}(\Alg\LLL )$ if and only if ${\rm{R_1}}(\Alg\LLL_{0})\subset {\rm{I_d}}(\Alg\LLL_{0})$.
\end{prop}

\begin{proof} We only prove the case when $\LLL=\LLL_{n_0}$ since the case when $\LLL=\LLL_{\infty}$ is similar. 

($\Rightarrow$) Let $x\otimes f\in\Alg\LLL_{0}$ be a rank 1 operator in $\Alg\LLL_{0}$. Then we have
$$
T=\left(\begin{array}{ccccc}
x\otimes f&-\frac{a_{1}}{a_{2}}x\otimes f&0&\cdots&0\\
0&0&0&\cdots&0\\
\vdots&\vdots&\vdots&\ddots&\vdots\\
0&0&0&\cdots&0
\end{array}\right)\in\Alg\LLL.
$$
By the hypothesis, there is some $k\in\Nb$, $l_{1},\cdots,l_{k}\in\Cb$ and $T_{1},\cdots,T_{k}\in{\rm{I_d}}(\Alg\LLL)$ such that $T=l_{1}T_{1}+\cdots+l_{k}T_{k}$. It follows that $T_{i}F_{1}\in\Alg\LLL_{0}$ is idempotent operator for any $i=1,\cdots,k$ and $x\otimes f=l_{1}T_{1}F_{1}+\cdots+l_{k}T_{k}F_{k}$. Therefore, $x\otimes f\in{\rm{I_d}}(\Alg\LLL_{0})$.

($\Leftarrow$) Let $x\otimes f\in\Alg\LLL$ be a rank 1 operator. We have
$$
(x\otimes f)(x\otimes f)=\langle x,f\rangle x\otimes f.
$$

If $\langle x,f \rangle\neq0$, then $x\otimes f$ is a nonzero multiple of an idempotent operator.

Suppose that $\langle x,f \rangle=0$. Since $x\otimes f\in\Alg\LLL$, there is some $2\leq k\leq n_{0}-1$ or $P\in\LLL_{0}$ such that $x\in F_{k}$ and $f\in (F_{k})_{-}^{\perp}$, or alternatively, $x\in E_{11}\otimes P$ and $f\in (E_{11}\otimes P)_{-}^{\perp}$, or else$x\in Q_{\zeta}$ and $f\in (Q_{\zeta})_{-}^{\perp}$.

Case 1: If $x\in E_{11}\otimes P$ and $f\in (E_{11}\otimes P)_{-}^{\perp}$. Let
$$
x\otimes f=\left(\begin{array}{cccc}
T_{11}&T_{12}&\cdots&T_{1,n_{0}}\\
0&0&\cdots&0\\
\vdots&\vdots&\ddots&\vdots\\
0&0&\cdots&0
\end{array}\right)
$$
and
$$
H=\left(\begin{array}{ccccc}
T_{11}&-\frac{a_{1}}{a_{2}}T_{11}&0&\cdots&0\\
0&0&0&\cdots&0\\
\vdots&\vdots&\vdots&\ddots&\vdots\\
0&0&0&\cdots&0
\end{array}\right).
$$
Then $T_{11}\in\Alg\LLL_{0}$ is a rank 1 operator and $H\in\Alg\LLL$. Since ${\rm{R_1}}(\Alg\LLL_{0})\subset {\rm{I_d}}(\Alg\LLL_{0})$, then there is some $s\in\Nb$ such that $k_{1},\cdots,k_{s}\in\Cb$ and idempotents $S_{1},\cdots,S_{s}\in\Alg\LLL_{0}$ such that $T_{11}=k_{1}S_{1}+\cdots+k_{s}S_{s}$. Let
$$
H_{i}=\left(\begin{array}{cccc}
S_{i}&-\frac{a_{1}}{a_{2}}S_{i}&\cdots&0\\
0&0&\cdots&0\\
\vdots&\vdots&\ddots&\vdots\\
0&0&\cdots&0
\end{array}\right), \ \forall i=1,\cdots,s.
$$
Then $H_{i}\in\Alg\LLL$ is idempotent operator for all $i=1,\cdots,s$ and
$$
H=k_{1}H_{1}+\cdots+k_{s}H_{s}\in{\rm{I_d}}(\Alg\LLL).
$$
On the other hand, we have
$$
x\otimes f-H=\left(\begin{array}{ccccc}
0&\frac{a_{1}}{a_{2}}T_{11}+T_{12}&T_{13}&\cdots&T_{1,n_{0}}\\
0&0&0&\cdots&0\\
\vdots&\vdots&\vdots&\ddots&\vdots\\
0&0&0&\cdots&0
\end{array}\right)\in\Alg\LLL.
$$
Hence $E_{11}\otimes I-\frac{a_{1}}{a_{2}}E_{12}\otimes I+(x\otimes f-H), \ E_{11}\otimes I-\frac{a_{1}}{a_{2}}E_{12}\otimes I-(x\otimes f-H)\in\Alg\LLL$ are idempotent operators, and $x\otimes f-H\in{\rm{I_d}}(\Alg\LLL)$. Therefore, $x\otimes f$ is contained in the linear span of the idempotent operators of $\Alg\LLL$.

Case 2: If $x\in F_{k}$ and $f\in (F_{k})_{-}^{\perp}$. Without loss of generality, we assume that $E_{kk}x\neq0$, otherwise, we may scrutinize the scenario of $k-1$, decrementing $k$ until it attains the value $k=1$, at which juncture it falls under Case 1.

Since $F_{k}\nleqslant (F_{k})_{-}=F_{k-1}\vee Q_{\zeta}$, there is some $g\in (F_{k})_{-}^{\perp}$ such that $\langle x,g \rangle=1$. Otherwise, $x\in (F_{k})_{-}\wedge F_{k}=(F_{k-1}\vee Q_{\zeta})\wedge F_{k}=F_{k-1}$, which is a contradiction. It is clearly that $x\otimes (f-g),x\otimes (f+g)\in\Alg\LLL$ are two nonzero multiple of idempotent operators and $x\otimes f=\frac{1}{2}x\otimes (f-g)+\frac{1}{2}x\otimes (f+g)$ is contained in the linear span of the idempotent operators of $\Alg\LLL$.

Case 3: If $x\in Q_{\zeta}$ and $f\in (Q_{\zeta})_{-}^{\perp}=E_{n_{0},n_{0}}$, then there is some $g\in E_{n_{0},n_{0}}$ such that $\langle x,g\rangle=1$. Thus $x\otimes (f-g),x\otimes (f+g)\in\Alg\LLL$ are two nonzero idempotent operators and $x\otimes f$ is contained in the linear span of the idempotent operators of $\Alg\LLL$.

\end{proof}

\begin{lem}\label{finitedecomposable}(see \cite{Chen}, \cite{Yang})
$\Alg\LLL$ is decomposable if and only if $\Alg\LLL_{0}$ is decomposable.
\end{lem}

\begin{lem}\label{finiteidempotent}(see \cite{LoHL}) Suppose $\mathscr{H}$ is a complex Hilbert space. 
Let $\AAA$ be a unital subalgebra of $\BBB(\mathscr{H})$ and $\delta$ a local derivation from $\AAA$ into $\BBB(\mathscr{H})$. Then
$$
\delta(MAN)=M\delta(AN)+\delta(MA)N-M\delta(A)N
$$
for any $A,M,N\in\AAA$ and $M,N$ idempotents.
\end{lem}

\begin{lem}\label{I-0}  Suppose $\mathscr{H}$ is a complex Hilbert space. 
Let $\delta$ be a local derivation from $\Alg\LLL$ into $\BBB(\mathscr{H})$. Then $\delta(I)=0$.
\end{lem}

The proof is well known. For completeness, we give the details. 

\begin{proof}
By calculation, we have
$$
\delta(I)=\delta_{I}(I)=\delta_{I}(I\cdot I)=\delta_{I}(I)I+I\delta_{I}(I)=2\delta_{I}(I),
$$
thus $\delta(I)=0$.
\end{proof}

\begin{cor}\label{finite rank}
Let $\Alg\LLL_{0}$ be a decomposable subspace lattice algebra and ${\rm{R_1}}(\Alg\LLL_{0})\subset {\rm{I_d}}(\Alg\LLL_{0})$. Let $\delta$ be a local derivation from $\Alg\LLL$ into $\BBB(\mathscr{K})$. Then
$$
\delta(MAN)=M\delta(AN)+\delta(MA)N-M\delta(A)N
$$
for any $A,M,N\in\Alg\LLL $ and $M,N$ are finite rank operators.
\end{cor}

\begin{proof} We only prove the case when $\LLL=\LLL_{n_0}$. 

Given $n\in\Nb$ and $T\in\Alg\LLL$ be a rank $n$ operator. Since $\Alg\LLL_{0}$ is decomposable, it follows from Lemma \ref{finitedecomposable} that $\Alg\LLL$ is decomposable. Then there are some rank 1 operators $x_{i}\otimes f_{i}\in\Alg\LLL$ such that $T=\sum_{i=1}^{n}x_{i}\otimes f_{i}$. Therefore for any finite rank operators $M,N\in\Alg\LLL$, it follows from Lemma \ref{finiteidempotent}, Lemma \ref{rank 1} and the linear of local derivations that
$$
\delta(MAN)=M\delta(AN)+\delta(MA)N-M\delta(A)N
$$
for any $A,M,N\in\Alg\LLL$ and $M,N$ are finite rank operators.
\end{proof}

\begin{lem}\label{AT}(see \cite{Chen}, \cite{Yang})
Let $T\in\Alg\LLL$ be nonzero element. If $P\nless P_{-}^{\LLL_{0}}$ for any nonzero projection $P\in\LLL_{0}$, then there is a rank 1  operator $A\in\Alg\LLL$ such that $AT\neq0$.
\end{lem}

\begin{lem}\label{TB}
Let $T\in\Alg\LLL$ be a nonzero operator. Then there is some rank 1 operator $B\in\Alg\LLL$ such that $TB\neq0$.
\end{lem}

\begin{proof} We only prove the case when $\LLL=\LLL_{n_0}$. 

Since $I=F_{n_{0}-1}\vee Q_{\zeta}$, then we have $TF_{n_{0}-1}\neq0$ or $TQ_{\zeta}\neq0$. 
If $TF_{n_{0}-1}\neq0$, there is some $z\in F_{n_{0}-1}$ and $f\in (F_{n_{0}-1})_{-}^{\perp}$ such that $Tz\neq0$ and $z\otimes f\in\Alg\LLL$. Let $B=z\otimes f$. Then $T(z\otimes f)=(Tz)\otimes f\neq 0$. If $TQ_{\zeta}\neq0$, then there is some $\xi\in Q_{\zeta}$ and $g\in(Q_{\zeta})_{-}^{\perp}$ such that $T\xi\neq0$ and $x\otimes g\in\Alg\LLL$. Let $B=\xi\otimes g$. Then $T(\xi\otimes g)=(T\xi)\otimes g\neq 0$. This proved the Lemma.
\end{proof}

\begin{thm}
Let $\Alg\LLL{_0}$ be a decomposable algebra and ${\rm{R_1}}(\Alg\LLL_{0})\subset {\rm{I_d}}(\Alg\LLL_{0})$. If $P\nless P_{-}^{\LLL_{0}}$ for any nonzero projection $P\in\LLL_{0}$, then every local derivation from $\Alg\LLL$ into $\BBB(\mathscr{K})$ is a derivation.
\end{thm}

\begin{proof} We only prove the case when $\LLL=\LLL_{n_0}$. 
It follows from Proposition \ref{finite rank} that
\begin{equation}\label{three}
\delta(MAN)=M\delta(AN)+\delta(MA)N-M\delta(A)N
\end{equation}
for any $A,M,N\in\Alg\LLL$ and $M,N$ are finite rank operators. Let $A=I$ in Equation (\ref{three}). Then we have
\begin{equation}\label{two}
\delta(MN)=M\delta(N)+\delta(M)N
\end{equation}
for any finite rank operators $M,N\in\Alg\LLL$. Replaced $M$ by $MA$ in Equation (\ref{two}), we obtain
\begin{equation}\label{two to there}
\delta(MAN)=MA\delta(N)+\delta(MA)N
\end{equation}
for any finite rank operators $M,N\in\Alg\LLL$. Combining with Equation (\ref{three}) and (\ref{two to there}), we have
$$
M\delta(AN)=M\delta(A)N+MA\delta(N).
$$
By Lemma \ref{AT}, we have
$$
\delta(AN)=\delta(A)N+A\delta(N)
$$
for any $A,N\in\Alg\LLL$ and $N$ is finite rank operator. It follows that
\begin{align*}
\delta(AB)N&=\delta(ABN)-AB\delta(N)\\
&=\delta(A)BN+A\delta(BN)-AB\delta(N)\\
&=\delta(A)BN+A\delta(B)N
\end{align*}
for any $A,B,N\in\Alg\LLL$ and $N$ is finite rank operator. It follows from Lemma \ref{TB} that
$$
\delta(AB)=\delta(A)B+A\delta(B)
$$
for any $A,B\in\Alg\LLL$. Hence $\delta$ is a derivation.
\end{proof}

\end{document}